\documentclass{amsart}

\usepackage[english]{babel}
\usepackage[latin1]{inputenc}
\usepackage{amsmath,amssymb,amsthm,epsfig}
\usepackage{mathrsfs}
\usepackage{latexsym}
\usepackage{ifthen,latexsym}
\usepackage{url}
\input{xypic.tex} \xyoption{all}

\newtheorem{theorem}{Theorem}
\newtheorem{definition}{Definition}
\newtheorem{proposition}[definition]{Proposition}

\newtheorem{corollary}[definition]{Corollary}

\newcommand{\id}{{\mathbf 1}}
\newcommand{\ox}{\mathcal{O}_{X}}
\newcommand{\opn}{\mathcal{O}_{\mathbb{P}^n}}
\DeclareMathOperator{\codim}{{codim}}

\newcommand{\marginnote}[1]{\ifthenelse{\isodd{\thepage}}{\normalmarginpar}
{\reversemarginpar}\marginpar{\fbox{\parbox{24mm}{\sloppy\footnotesize
#1}}}}

 \newcommand{\pn}{{\mathbb{P}^n}}

\newcommand{\F}{{\mathcal{F}}}
\begin{document}


\title{Bounds for sectional genera of varieties invariant under Pfaff fields}

\author{Maur\'{\i}cio Corr\^ea  Jr.}
\address{Departamento de Matem\'atica, Universidade Federal de Vi\c{c}osa-UFV, Avenida P.H. Rolfs, 36571-000 Brazil}
\email{mauricio.correa@ufv.br}
\thanks{Partially supported by CNPq.}

\author{Marcos Jardim}
\address{Instituto de Matem\'atica, Estat\'{\i}stica e Computa\c{c}\~ao Cient\'{\i}fica \\ Universidade Estadual de Campinas \\ Rua S\'ergio Buarque de Holanda, 651 \\ Campinas, SP, Brazil \\ CEP 13083-859}
\email{jardim@ime.unicamp.br}
\thanks{Partially supported by the CNPq grant number 305464/2007-8 and the FAPESP grant number 2005/04558-0.}


\begin{abstract}
We establish an upper bound for the sectional genus of varieties
which are invariant under Pfaff fields on projective spaces.
\end{abstract}

\maketitle

\section{Introduction}

In \cite{Pa} P. Painlev\'e  asked the following question: \emph{``Is
it possible to recognize the genus of the general solution of an
algebraic differential equation in two variables which has a
rational first integral?}" In \cite{N}, Lins Neto  has constructed
families of foliations with fixed degree and local analytic type of
the singularities where foliations with rational first integral of
arbitrarily large degree appear. In other words, such families show
that Painlev\'e's question has a negative answer.

However, one can obtain an affirmative answer to Painlev\'e's question
provided some additional hypotheses are made. The problem of
bounding the genus of an invariant curve in terms of the degree of a
foliation on $\mathbb{P}^n$ has been considered by several authors,
see for instance \cite{CN,E-K}. In \cite{C-C-F}, Campillo, Carnicer
and de la Fuente showed that if  $C$ is a reduced curve which is invariant by a one-dimensional foliation  $\F$ on $\mathbb{P}^n$ then
\begin{equation}\label{ccf}
\frac{2p_a(C)-2}{\deg(C)}\leq \deg(\F)-1+a,
\end{equation}
where $p_a(C)$ is the arithmetic genus of $C$ and $a$ is an integer
obtained from the concrete problem of imposing singularities
 to projective hypersurfaces. For instance, if $C$ has only nodal singularities then $a=0$, and thus formula $(\ref{ccf})$
 follows from \cite{Fu}. This bound has been improved by Esteves and Kleiman in \cite{E-K}.

Painlev\'e's question is related to the problem posed by Poincar\'e in
\cite{Po} of bounding the degree of algebraic solutions of an
algebraic differential equation on the complex plane. Nowadays, this
problem is known as $\emph{Poincar\'e's Problem}$. Many mathematicians
have been working on it and on some of its generalizations, see for
instance the papers by Cerveau and Lins Neto \cite{CN}, Carnicer
\cite{C}, Pereira \cite{Pe}, Soares \cite{S}, Brunella and Mendes
\cite{B-M}, Esteves and Kleiman \cite{E-K}, Cavalier and Lehmann
\cite{C-L}, and Zamora \cite{Za}.

In \cite{E-K}, Esteves and Kleiman extended Jouanolou's work on
algebraic Pfaff systems on a nonsingular scheme $V$. Essentially,
an algebraic Pfaff system  is a singular distribution. More
precisely, an algebraic Pfaff system of rank $r$ on a nonsingular
scheme $X$ of pure dimension $n$ is, according to Jouanolou \cite[pp.
$136$-$38$]{J}, a nonzero map $u:E\rightarrow \Omega_{X}^1$ where
$E$ is a locally free sheaf of constant rank $r$ with $1\leq r\leq
n-1$. Esteves and Kleiman introduced the notion of a \emph{Pfaff
field} on $V$, which is a nontrivial sheaf map $\eta:\Omega_{V}^k\to
L$, where $L$ is a invertible sheaf on $V$, and the integer $1\leq
k\leq n-1$ is called the rank of $\eta$. A subvariety $X\subset V$
is said to be invariant under $\eta$ if the map $\eta$ factors
through the natural map $\Omega^k_{V}|_{X}\to\Omega^k_X$. A Pfaff
system on $V$  induces, via exterior powers and the perfect pairing
of differential forms, a Pfaff field on $V$. However, the converse
is not true; see \cite[Section 3]{E-K} for more details.

In this paper, we establish new upper bounds for the sectional genera
of nonsingular projective varieties which are invariant under Pfaff
fields on $\mathbb{P}^n$.

First, we use the hypothesis of stability (in the
sense of Mumford--Takemoto) of the tangent bundle of $X$ to
establish an upper bound for the sectional genus in terms of
the degree and the rank of a Pfaff field.

More precisely, our first main result is the following. Let
$g(X,\mathcal{O}_{X}(1))$ denote the sectional genus of $X$ with
respect to the line bundle $\mathcal{O}_{X}(1)$ associated to
the hyperplane section.

\begin{theorem}\label{2}
Let $X$ be a nonsingular projective variety of dimension $m$ which
is invariant under a Pfaff field $\F$ of rank $k$ on
$\mathbb{P}^n$; assume that $m\geq k$. If the tangent bundle
$\Theta_{X}$ is stable, then
\begin{equation}\label{f2}
\frac{2g(X,\mathcal{O}_{X}(1))-2}{\deg(X)} \leq
\dfrac{\deg(\F)-k}{{m-1\choose k-1}}+m-1.
\end{equation}
\end{theorem}

To the best of our knowledge, this is the first time that the
stability of the tangent bundle is used to obtain such bounds.
Notice that the left-hand side of inequality (\ref{f2}) does not
change when we take generic linear sections $\mathbb{P}^l\subset\mathbb{P}^n$,
while the right-hand side gets larger, and so the bound becomes worse. This means that
the above result is a truly higher dimensional one.

Examples of projective varieties with stable tangent bundle are
Calabi--Yau \cite{T}, Fano \cite{Fa,Hw,PW,St} and complete
intersection \cite{PW,Sb} varieties.

In the critical case when the rank of Pfaff field $\F$ is equal to the
dimension of the invariant variety $X$, we show that one can substitute
for the stability condition the conditions of $X$ being Gorenstein and
smooth in codimension $1$, i.e. $\mathrm{codim}(Sing(X),X)\geq 2$.

\begin{theorem}\label{1}
Let $X\subset\mathbb{P}^n$ be a Gorenstein projective variety nonsingular in codimension $1$, which is invariant under a Pfaff field $\F$ on $\mathbb{P}^n$
whose rank is equal to the dimension of $X$. Then
\begin{equation}\label{f1}
\frac{2g(X,\mathcal{O}_{X}(1))-2}{\deg(X)}\leq \deg(\F)-1,
\end{equation}
\end{theorem}

This generalizes a bound obtained by Campillo, Carnicer and de la
Fuente in \cite[Theorem 4.1 (a)]{C-C-F}. As an application, we
improve upon a bound obtained by Cruz and Esteves \cite[Corollary
4.5]{C-E}, see Section \ref{ci}.

This note is organized as follows. First, in order to make this
presentation as self-contained as possible, we provide all the
necessary definitions in Section \ref{s1}. The proofs of our main
results along with some further consequences are given in Sections
\ref{p1} and \ref{p2}.


\section{Background material}\label{s1}

We work over the field of complex numbers. Let $(X,L)$ be a
Gorenstein projective variety $X$ of dimension $n$ equipped
with a very ample line bundle $L$; recall that, since $X$ is
Gorenstein, the canonical divisor $K_{X}$ is a Cartier divisor.

\begin{definition}\label{sg}
The \emph{sectional genus} of $X$ with respect to $L$, denoted
$g(X,L)$, is defined by the formula:
$$ 2g(X,L)-2=(K_{X}+(\dim(X)-1)L)\cdot L^{\dim(X)-1}.$$
\end{definition}

This quantity has the following geometric interpretation. Suppose
that $X$ is nonsingular, and let $H_1,\dots,H_{n-1}$ be general
elements in the linear system $|L|$.  By Bertini's Theorem,
the curve $X_{n-1}=H_1\cap\cdots \cap H_{n-1}$ is nonsingular. Then
$g(X,L)$ coincides with the geometric genus of $X_{n-1}$, see
\cite[Remark 2.5]{Fk}.

\begin{definition}
Let $(V,L)$ be a nonsingular polarized algebraic variety.  A
Pfaff field $\F$ of rank $k$ on $V$ is a nonzero global
section of $\bigwedge^k \Theta_{V}\otimes N$, where
$\Theta_{V}$ is the tangent bundle and $N$ is a line bundle, where
$0<k<n$. The degree of $\F$ with respect to $L$ is defined by the
formula $\deg_L(\F)=\deg_L(N)+k\deg_L(L)$, where the degree of a
line bundle $N$ relative to $L$ is given by $\deg_L(N)=N\cdot
L^{\dim(V)-1}$.
\end{definition}

Since the ambient space is nonsingular, our definition is equivalent
to the one introduced in \cite[Section 3]{E-K}. In fact, since
$\bigwedge^k\Theta_{V}\otimes N \simeq \mathcal{H}om(\Omega^k_{V},N)
\simeq \mathcal{H}om(N^*,\bigwedge^k\Theta_{V})$, a Pfaff field can
also be regarded either as a map $\xi_{\F}:N^*\rightarrow\bigwedge^k \Theta_{V}$
or as a map $\xi_{\F}^\vee:\Omega^k_{V}\rightarrow N$. The present definition emphasizes the existence of a global section of $\bigwedge^k \Theta_{V}\otimes N$, which will play a central role in our arguments.

\begin{definition}
The \emph{singular set} of $\F$ is given by
\begin{center}
$Sing(\F)=\{x\in V;\ \xi_{\F}(x) \ \ $is not injective$\}=\{x\in V;\
\xi_{\F}^\vee(x) \ \ $is not surjective$\}$.
\end{center}
\end{definition}

For instance, a Pfaff field of rank $k$ on $\mathbb{P}^n$ is a
section of
$\bigwedge^k\Theta_{\mathbb{P}^n}\otimes\mathcal{O}_{\mathbb{P}^n}(s)$,
and $\deg_{\opn(1)}(\F)=s+k$.

More generally, if $\mathrm{Pic}(V)\simeq\mathbb{Z}$ and
$L:=\mathcal{O}_{V}(1)$ is the positive generator of
$\mathrm{Pic}(V)$, then a Pfaff field of rank $k$ on $V$ is a
section of $\bigwedge^k\Theta_{V}\otimes\mathcal{O}_{V}(s)$,
for some $s\in\mathbb{Z}$.
Thus, $\deg_L(\F)=(s+k)\deg(V)$, where $\deg(V)=\deg_L(L)$. If we define
$d_{\F}:=s+k$ we have
$\deg_L(\F)=d_{\F}\cdot\deg(V).$

Alternatively, a Pfaff field can also be defined as a
global section of $\Omega^{n-k}_V\otimes N'$, where
$N'=N\otimes K_V^{-1}$. If $V$ is nonsingular, this definition is
equivalent to the one above.

Let $X\subset V$ be a closed subscheme of dimension larger than or
equal to the rank of a Pfaff field $\F$. Following
\cite[Section 3]{E-K}, we introduce the following definition.

\begin{definition}
We say $X$ is invariant under $\F$ if  $X\not\subset Sing(\F)$ and
there exists a morphism of sheaves $\phi:\Omega_{X}^k\rightarrow
N|_{X}$ such that the following diagram
$$
\xymatrix{
\Omega_{V}^k|_X  \ar[d]\ar[r]^{\xi_{\F}^{\vee}|_X} & N|_{X} \\
\Omega_{X}^k  \ar[ru]^\phi & }
$$
commutes.
\end{definition}

Applying the functor $\mathcal{H}om(\cdot, \mathcal{O}_{X})$ to the
above diagram, we get the following commutative diagram:
$$ \xymatrix{
N^*|_{X} \ar[d]^{\phi^\vee} \ar[dr]^{\xi_{\F}|_X} & \\
(\Omega_{X}^k)^{\vee} \ar[r] & \bigwedge^k \Theta_{V}|_{X} }. $$
Therefore, $X$ is invariant under $\F$ if $\xi_{\F}|_X$ induces a nonzero
global section of $(\Omega_{X}^k)^{\vee}\otimes N|_{X}$.

Our two main results are concerned only with the case when $V=\pn$;
but we would like to conclude this section with two general
propositions.

Let $E$ be a torsion-free sheaf on $V$. The ratio $\mu_L(E)=\deg_L(E)/{\rm rk}(E)$
is called the slope of $E$, where $\deg_L(E)=\deg_L((\Lambda^rE)^{\vee\vee})$ and
$r={\rm rk}(E)$. Recall that $E$ is \emph{semistable} (in the sense of Mumford--Takemoto)
if every torsion-free subsheaf $E'$ of $E$ satisfies $\mu_L(E')\le\mu_L(E)$. Furthermore, $E$ is \emph{stable} if the strict inequality is satisfied for proper subsheaves.
Further details can be found in \cite[Sections V.6 and V.7]{Ko}.

\begin{proposition}
If $\Theta_{V}$ is stable, then the following inequality holds:
$$ \deg_L(\F) \ge {\rm rk}(\F)\left( \deg_L(V)+\frac{\deg_L(K_V)}{\dim(V)} \right) . $$
\end{proposition}

If $V=\pn$ the above inequality becomes $\deg(\F)\ge0$. Bott's
formula \cite[page 8]{OSS} implies the existence of a rank $k$ Pfaff
field of degree $0$ for each $k$, hence in this case the bound given
above is sharp.

\begin{proof}
The stability of $\Theta_{V}$ implies that $\bigwedge^k\Theta_{V}$
is semistable with slope equal to $k\mu_L(\Theta_{V})$
\cite[Corollary 1.6]{AO}. As observed above, a Pfaff
field $\F$ of rank $k$ induces a map $\xi_{\F}:N^* \to
\bigwedge^k\Theta_{V}$, so from the semistability of
$\bigwedge^k\Theta_{V}$ we conclude that $-\deg_L(N)\le
k\mu_L(\Theta_{V})=-k\deg_L(K_V)/\dim(V)$. The stated inequality
follows easily.
\end{proof}

If $D$ is a divisor on an algebraic variety $V$
with $\mathrm{Pic}(V)\simeq\mathbb{Z}$, then
$\mathcal{O}_{V}(D)=\mathcal{O}_{V}(d_{D})$, for some $d_{D}\in
\mathbb{Z}$. In this case, we denote $\kappa(V)=d_{K_{V}}$.

\begin{proposition}\label{completa2}
Let $V$ be a $n$-dimensional nonsingular algebraic variety with
$\mathrm{Pic}(V)\simeq\mathbb{Z}$.
Let $X$ be a $k$-dimensional nonsingular complete intersection of  hypersurfaces $D_1,\dots,D_{n-k}$ on
$V$. If $X$ is invariant under a Pfaff field $\F$ of rank $k$ on $V$, then
$$ d_{D_1}+\cdots+d_{D_{n-k}}\leq d_{\F}-k-\kappa(V).$$
\end{proposition}

\begin{proof}
Since $X$ is invariant by $\F$ we have that
$H^0(X,\bigwedge^k\Theta_X\otimes\mathcal{O}_{V}(d_{\F}-k)|_{X})\neq\{0\}$,
then $\deg(\bigwedge^k\Theta_X\otimes
\mathcal{O}_{V}(d_{\F}-k)|_{X})\geq 0$. Let $\mathcal{O}_{V}(D_i)$
be the line bundle associated to the hypersurface $D_i$,
$i=1,\dots,n-k$. We have the following adjunction formula
$$
\bigwedge^k\Theta_X=\bigwedge^n\Theta_V|_{X}\otimes
\mathcal{O}_{V}(-D_1)|_{X}\otimes\cdots\mathcal{O}_{V}(-D_{n-k})|_{X}.
$$
Therefore
$\bigwedge^k\Theta_X=\mathcal{O}_{V}(-\kappa(V)-d_{D_1}-\cdots-d_{D_{n-k}})|_{X}$,
thus
$$
\deg(\mathcal{O}_{V}(d_{\F}-k-\kappa(V)-d_{D_1}-\cdots-d_{D_{n-k}})|_{X})=\deg(\bigwedge^k\Theta_X\otimes
\mathcal{O}_{V}(d_{\F}-k)|_{X})\geq 0.
$$
\end{proof}


\section{Proof of Theorem \ref{2}}\label{p2}

We recall that the stability of $\Theta_{X}$ implies that
$\bigwedge^k\Theta_{X}$ is semistable. Since $X$ is invariant under
$\F$, we can conclude that
$H^0(X,\bigwedge^k\Theta_X\otimes\mathcal{O}_{X}(d-k))\neq\{0\}$,
with  $d=\deg(\F)$. It then follows from the semistability of
$\bigwedge^k\Theta_X$ that
$\bigwedge^k\Theta_X\otimes\mathcal{O}_{X}(d-k)$ is also semistable,
thus
\begin{equation}\label{grau1}
\deg(\bigwedge^k\Theta_X\otimes\mathcal{O}_{X}(d-k))\geq 0.
\end{equation}
On the other hand, note that
\begin{equation}\label{grau2}
\deg(\bigwedge^k\Theta_X)=-{\dim(X)-1\choose k-1}\deg(K_{X}).
\end{equation}
Let $i:X\rightarrow \mathbb{P}^n$ be the embedding, and set, as
usual, $\mathcal{O}_{X}(1)=i^*\mathcal{O}_{\mathbb{P}^n}(1)$. Now,
we consider the following difference, using \ref{grau2}:
$$ (2g(X,\mathcal{O}_{X}(1))-2) - \left[ \frac{\mathcal{O}_{X}(d-k)}{{m-1\choose k-1}
}+(m-1)\mathcal{O}_{X}(1)\right]\cdot\mathcal{O}_{X}(1)^{m-1}= $$
$$ -\left(-K_{X}+ \frac{\mathcal{O}_{X}(d-k)}{{m-1\choose k-1}}\right)\cdot\mathcal{O}_{X}(1)^{m-1} =
-\frac{\deg(\bigwedge^k\Theta_X\otimes\mathcal{O}_{X}(d-k))}{{m-1\choose
k-1}}.
$$

It follows from (\ref{grau1}) that the difference must be less
than or equal to zero, hence
$$
\begin{array}{ccl}
  2g(X,\mathcal{O}_{X}(1))-2 & \leq & \left[ \frac{\mathcal{O}_{X}(d-k)}{{m-1\choose k-1}}+(m-1)
  \mathcal{O}_{X}(1)\right]\cdot\mathcal{O}_{X}(1)^{m-1}
\\
  \\
  & \leq & \deg(X)\left( \dfrac{d-k}{{m-1\choose k-1}}+m-1\right).
\end{array}
$$
This completes the proof of Theorem \ref{2}.

Let us now consider applications of Theorem \ref{2} to a few
particular cases. First, specializing to the case when the invariant
variety is Fano with Picard number one, i.e., $\deg(K_X)<0$ and
$\rho(X)=rank(NS(X))=1$, where $NS(X)$ is the N\'eron--Severi group of
$X$.

\begin{corollary}
Let $X$ be a nonsingular Fano variety, with Picard number one, and let
$\ox(1):=K_{X}^{-1}$. If $X$ is invariant under a Pfaff field $\F$ of
rank $k=\dim(X)$, then
$$
\deg_{K_X^{-1}}(X)\leq k^k(\deg(\F)+2)^k,
$$
where $\deg_{K_X^{-1}}(X)$ is the degree of $X$ with respect to the
anticanonical polarization.
\end{corollary}
\begin{proof}
Indeed, in this case we have
$$
2g(X,K_X^{-1})-2=(k-2)\deg_{K_X^{-1}}(X).
$$
Thus, it follows from Theorem \ref{2} that $k\leq \deg(\F)+1$. On
the other hand, it follows from \cite{Na} that $d(X)\leq k+1$ and
$\deg_{K_X^{-1}}(X)\leq (d(X)k)^k$ , where $d(X)$ is the least
positive integer $d$ for which $X$ can be covered by rational
curves of (anticanonical) degree at most $d$, see \cite[Subsection
1.3]{Na}.
\end{proof}

Finally, we also consider the case when the invariant variety is
Calabi--Yau, i.e. $K_X=0$.

\begin{corollary}
If $X$ is Calabi--Yau and invariant by $\F$ then ${\rm rk}(\F)\leq\deg(\F)$.
\end{corollary}

In other words, Pfaff fields of small degree do not
admit invariant Calabi--Yau varieties.


\section{Proof of Theorem \ref{1}}\label{p1}

First, let us briefly recall the construction of the so-called
\emph{canonical map} $\gamma_{X}:\Omega^k_{X}\rightarrow\omega_{X},$
where $\omega_{X}$ is the dualizing sheaf of $X$, as it was done in
\cite[Section 3]{C-E}.

Let $X$ be a reduced projective variety of pure dimension $k$, and
let $X_1,\dots, X_s$ be its irreducible components. For each
$i=1,\dots,s$, consider Kunz's sheaf $\widetilde{\omega}_{X_i}$
of regular differential forms of $X_i$, see \cite{Kz}. By
definition, the canonical map $\gamma_{X}$ is the composition
$$
\begin{array}{ccccccccc}
  \Omega^k_{X} & \stackrel{\widetilde{\tau}}{\longrightarrow} & \bigoplus_{i=1}^{s}
  \Omega^k_{X_i} & \stackrel{(\gamma_{1},\dots,\gamma_{s})}{\longrightarrow}
  & \bigoplus_{i=1}^{s} \widetilde{\omega}_{X_i} & \stackrel{(\zeta_{1},\dots,\zeta_{s})}{\longrightarrow} &
  \bigoplus_{i=1}^{s} \omega_{X_i} &
  \stackrel{\tau}{\longrightarrow}&\omega_{X},
\end{array}
$$
where $\widetilde{\tau}$ and $\tau$ are the maps induced by
restriction, for each $i=1,\dots,s$ the map
 $\gamma_{i}:\Omega^k_{X_i}\rightarrow \widetilde{\omega}_{X_i}$ is the canonical class of $X_i$, constructed by Lipman  in \cite{L}, which is an isomorphism on
the nonsingular locus of $X_i$. Moreover, $\zeta_{i}:
\widetilde{\omega}_{X_i}\rightarrow \omega_{X_i}$ is a isomorphism
on $X_{i}$, since it follows from \cite[Theorem 0.2B]{L} that
$\widetilde{\omega}_{X_i}$ is dualizing. Therefore, $\gamma_{X}$ is
an isomorphism on the nonsingular locus  $X_0:=X-Sing(X)$. Thus the
map
$$ \widetilde{\gamma_X}=\gamma_{X}^{\vee}\otimes\id_{\ox(d-k)}:
\omega_{X}^{\vee}\otimes\mathcal{O}_{X}(d-k) \rightarrow
(\Omega^k_{X})^{\vee}\otimes \mathcal{O}_{X}(d-k)$$ is also an
isomorphism when restricted to $X_0$.

Now assume that $X\subset\mathbb{P}^n$ is a Gorenstein variety of
pure dimension $k$ such that $\codim(Sing(X),X)\geq 2$. Then the
sheaf $\omega_{X}^{\vee}$ is locally-free, hence, in particular,
reflexive. Moreover, from \cite[Proposition 5.21]{Ko}, we also
conclude that $\omega_{X}^{\vee}$ is normal.

If $X$ is invariant under a Pfaff field $\F$ on
$\mathbb{P}^n$ of rank $k$ and degree $d$, then we have a nonzero
global section $\zeta_{\F}$ of $(\Omega^k_{X})^{\vee}\otimes
\mathcal{O}_{X}(d-k)$; consider its restriction
$\zeta_{\F,0}=\zeta_{\F}|_{X_0}$ to $X_0$. Composing it with the the
inverse of $\widetilde{\gamma_X}|_{X_0}$, the restriction of the map
$\widetilde{\gamma_X}$ to $X_0$, we obtain a section
$$ \widetilde{\gamma_X}|_{X_0}(\zeta_{\F,0})\in H^0(X_0,\omega_{X}^{\vee}\otimes \mathcal{O}_{X}(d-k)|_{X_0}). $$
However, $\omega_{X}^{\vee}\otimes \mathcal{O}_{X}(d-k)|_{X_0}$ is a
normal sheaf, so the above section extends to a global section of
$\omega_{X}^{\vee}\otimes \mathcal{O}_{X}(d-k)$. In particular,
$H^0(X,\omega_{X}^{\vee}\otimes \mathcal{O}_{X}(d-k))\neq\{0\}$,
therefore
\begin{equation}\label{grau}
\deg(\omega_{X}^{\vee}\otimes \mathcal{O}_{X}(d-k))\geq0.
\end{equation}

Let $K_{X}$ be a Cartier divisor such that
$\mathcal{O}_{X}(K_{X})=\omega_{X}$.

Now, consider the following difference
$$ (2g(X,\mathcal{O}_{X}(1))-2) -[ \mathcal{O}_{X}(d-k)+(k-1)\mathcal{O}_{X}(1)]\cdot\mathcal{O}_{X}(1)^{k-1} = $$
$$ -\left(K_{X}^{-1}+ \mathcal{O}_{X}(d-k)\right)\cdot\mathcal{O}_{X}(1)^{k-1} =
-\deg(\omega_{X}^{\vee}\otimes\mathcal{O}_{X}(d-k))\leq 0. $$


\section{Complete intersection invariant varieties}\label{ci}

We specialize to the case when the invariant variety $X$ is a complete intersection.

First, we notice that the inequality of Theorem \ref{2} is not sharp
in general. To see this, let $X$ be a nonsingular complete
intersection variety of dimension $m$ and multidegree
$(d_1,\dots,d_{n-m})$, which is invariant under a $k$-dimensional
Pfaff field $\F$ on $\mathbb{P}^n$; assume that $m\ge k$. It follows
from \cite[Corollary 1.5]{PW} that $\Theta_X$ is stable and one can
apply Theorem \ref{2} to obtain the following inequality:
$$
d_1+\cdots+d_{n-m}\leq \dfrac{\deg(\F)-k}{{m-1\choose k-1}}+n+1.
$$
Setting $m=n-1$ and $k=1$, the inequality reduces to $d_1 \le \deg(\F) + n$.
However, Soares has shown, under the same circumstances, that $d_1
\le \deg(\F) + 1$ \cite[Theorem B]{S}.

In the critical case $\dim(X)={\rm rank}(\F)$, Theorem \ref{1} gives us the following Corollary.

\begin{corollary}\label{completa}
Let $X$ be a $k$-dimensional complete intersection variety of multidegree $(d_1,\dots,d_{n-k})$ such that
either $X$ is nonsingular in codimension $1$. If $X$ is invariant under a Pfaff field $\F$ of rank
$k$ on $\mathbb{P}^n$, then
$$ d_1+\cdots+d_{n-k}\leq\deg(\F)+n-k+1.$$
\end{corollary}

\begin{proof}
From the adjunction formula for dualizing sheaves one obtains
$$ 2g(X,\mathcal{O}_{X}(1))-2 = \deg(X) \left( d_1+\cdots+d_{n-k} -n +k -2 \right). $$
By Theorem \ref{1}, this is less than or equal to
$(\deg(\F)-1)\deg(X)$, and the desired inequality follows easily.
\end{proof}

It follows from \cite[Corollary 4.5]{C-E} that if $X$ and $\F$ are
as above, then
$$
d_1+\cdots+d_{n-k}\leq \left\{
  \begin{array}{ll}
    \deg(\F)+n-k, & \hbox{if } \ \rho\leq 0  \\
\\
    \deg(\F)+n-k+\rho, & \hbox{if } \rho>0
  \end{array}
\right.
$$
where $\rho:= \sigma+n-k+1-d_1-\cdots-d_{n-k}$, with $\sigma$
denoting the Castelnuovo--Mumford regularity of the singular locus of
$X$. Therefore, Corollary \ref{completa} allows us to conclude that
if $X$ is nonsingular in codimension $1$, then one can take $\rho=1$, regardless
of $\sigma$.




\end{document}